\documentclass[12pt,psfig]{article}
\usepackage{}
\usepackage{graphics,epsfig,amssymb}
\usepackage{amsmath}
\usepackage{amsthm}
\usepackage{amsfonts}
\usepackage{amssymb,latexsym}
\usepackage[all]{xy}
\usepackage{graphicx}
\usepackage[mathscr]{eucal}
\usepackage{verbatim}
\usepackage{hyperref}
\usepackage{color}

\theoremstyle{plain}
    \newtheorem{thm}{Theorem}[section]
    \newtheorem{prop}{Proposition}[section]
    \newtheorem{lemma}{Lemma}[section]
    \newtheorem{rem}{Remark}[section]

\numberwithin{equation}{section}

\usepackage{graphicx}
\usepackage{pict2e}

\begin{document}
\begin{center}
{\bf\Large
Coexisting  steady-state solutions of a class of reaction-diffusion systems with different boundary conditions}

\vskip.20in

{Ningning~Zhu$^1$, Dongpo Hu$^1$, Huili Bi$^{2,*}$}

\vskip.10in

{\footnotesize \it
1 School of Mathematical Sciences, Qufu Normal University, Qufu, Shandong 273165, P.R. China\\
2 School of Statistics and Data Science, Qufu Normal University, Qufu, Shandong 273165, P.R. China\\}

\end{center}

\noindent\footnotetext{
$^{*}$ Corresponding author.\\
E-mail address: bihuili0810@126.com
}

\begin{abstract}
In this work, we investigate a reaction-diffusion system in which both species are influenced by self-diffusion.
Due to Hopf's boundary lemma, we obtain the boundedness of the classical solution of the system.
By considering a particular function,
we provide a complete characterization of the parameter ranges
such that coexisting solutions of the system do not exist under three boundary conditions.
Then based on the maximum principle,
a sufficient condition for the existence of constant coexisting solutions of the system
under Neumann boundary conditions was derived.
\end{abstract}

\vskip.10in
\noindent {\bf AMS(2020) Subject Classification} 35B50; 35J57.

\baselineskip=18pt


\section{Introduction}

In this work, we study the steady-state solutions of the following competing systems with cross-diffusions and self-diffusions,
\begin{align}\label{skt}
\left\{
\begin{aligned}
&\frac{\partial u}{\partial t}=\Delta[(d_{1}+a_{11}u+a_{12} v)u]+u(1-u-a_1v),&~~&x\in\Omega, t>0,\\
&\frac{\partial v}{\partial t}=\Delta[(d_{2}+a_{21}u+a_{22} v)v]+v(1-a_2u-v),&~~&x\in\Omega, t>0,\\
&\alpha_1u+\beta_1\frac{\partial u}{\partial\nu}=\alpha_2v+\beta_2\frac{\partial v}{\partial\nu}=0,&~~&x\in\partial\Omega, t>0,\\
&u(x,0)=u_{0}(x)\geq 0,~v(x,0)=v_{0}(x)\geq 0,&~~&x\in\Omega,
\end{aligned}
\right.
\end{align}
where $\Omega\subset \mathbb{R}^{n}(n\geq1)$ is a bounded domain with smooth boundary,
and satisfies the interior ball condition at any $x\in\partial\Omega$.
$u$ and $v$ are the densities of two competing species,
$\alpha_i, \beta_i$ and $a_{ij} (i,j=1,2)$ are nonnegative constants,
$\alpha_i$ and $\beta_i$ ($i=1,2$) cannot be zero at same time,
$a_{i}$ and $d_{i} (i=1,2)$ are all positive constants,
$a_{11}$ and $a_{22}$ stand for the self-diffusion pressures,
while $a_{12}$ and $a_{21}$ are the cross-diffusion pressures,
$a_1, a_2$ describe the inter-specific competitions,
and $d_{1}, d_{2}$ are their diffusion rates \cite{z2021}.

In order to analyze and describe the above reaction-diffusion model,
it is necessary to clarify the boundary conditions of the region.
When $\alpha_i=0~(i=1,2)$, we have $\frac{\partial u}{\partial\nu}=\frac{\partial v}{\partial\nu}=0$,
which is called Neumann boundary condition.
At this point, individuals who reach the boundary will be reflected back into the region without leaving,
meaning that the species is in an isolated environment.
When $\beta_i=0~(i=1,2)$, we have $u=v=0$.
These indicate that individuals who encounter the boundary cross it immediately and thereby
maintain the density on the boundary at zero \cite{cc2003}.
This means that the boundary $\partial\Omega$ can effectively absorb all individuals encountering it.
Therefore, this boundary is called absorbed, also known as Dirichlet boundary condition.
In population dynamics, the Dirichlet boundary condition can sometimes be considered a lethal boundary,
because it can be interpreted as meaning that all individuals who encounter $\partial\Omega$ die.
When $\alpha_i, \beta_i>0~(i=1,2)$, it is called Robin boundary condition in the mathematical literature.

The system \eqref{skt} was first introduced by Shigesada, Kawasaki and Teramoto in 1979 \cite{skt1979},
when they took a nonlinear dispersive force and an environmental potential function into consideration.
Since the introduction of the SKT model,
numerous experts and scholars have conducted extensive and in-depth research on it.
These results mainly include the existence, boundedness and global convergence of classical solutions,
the local and global existence of weak solutions,
the existence, nonexistence and stability of steady-state solutions,
the existence of traveling wave solutions, and so on.

In fact, the study of standard SKT models is quite difficult.
More experts and scholars are turning to the study of certain special forms of SKT models.
\cite{kim1984,d1987,shim2002,lw2015} investigated the existence and boundedness of smooth solutions for SKT model
without considering self-diffusion effects.
Among them, when the spatial dimension is 1 and $d_1=d_2$,
\cite{kim1984} proved the global existence of smooth solutions.
Under the same conditions, \cite{shim2002} obtained the uniform boundedness and convergence of smooth solutions.
In 2015, Lou and Winkler \cite{lw2015} used the comparison principle and Sobolev regularity theory to obtain
the global existence and uniform boundedness of smooth solutions on bounded convex domain
when the spatial dimension is less than 3 and $d_1=d_2$.
Taking self-diffusion into consideration,
\cite{hnp2015} proved the global existence of the unique smooth solution in any spatial dimension;
and \cite{lz2005} obtained the global existence of the unique classical solution by using Sobolev embedding theory
under the condition of $d_1=d_2$.

For research on steady-state solutions, one can refer to \cite{m1981,k1993,ln1996,ln1999,ltw2017}.
By bifurcations and singular perturbation methods,
\cite{m1981} proved  the existence of nonconstant positive steady-state solutions for systems on intervals.
In 1996, Lou and Ni \cite{ln1996} used the maximum principle and Lyapunov functional theory
to prove that in the weak competition case, if self-diffusion and/or cross-diffusion are relatively weaker than diffusion,
then there is still no nonconstant steady-state solution.
At the same time, they proved that in the weak competition case,
with one of the cross-diffusion pressures arbitrarily given but fixed,
it is expect to find non-constant steady-state solutions if the other cross-diffusion pressure is large enough \cite{ln1996}.
Without considering the influence of self-diffusion,
\cite{ln1999} obtained a sufficient condition such that the SKT model has no nonconstant steady-state solutions.
When $a_{21}=a_{22}=0$, Lou et al. \cite{ltw2017} provided the parameters ranges such that
the system has no nonconstant positive solutions for $a_{11}=0$ and $a_{11}\neq0$, respectively.

It is obvious to see that the above studies were conducted under Neumann boundary conditions,
which is also the most extensively studied scenario.
In addition, some scholars have also studied the SKT model under Dirichlet boundary conditions,
which can be found in \cite{d1991,ruan1996,wu2002,zt} and references therein.
The sufficient conditions for the existence of positive steady-state solutions of the system
under Dirichlet boundary conditions are given in \cite{ruan1996} using the fixed point theory
in the case of fixed or sufficiently large cross-diffusion coefficients, respectively.
The existence of steady-state solutions for a one-dimensional system
was studied using the singular perturbation method \cite{wu2002}.

In this work, we consider the following model, which indicates that
there are self-diffusions in both competing species and there is no cross-diffusion in either,
\begin{align}\label{equ}
\left\{
\begin{aligned}
&\Delta[(d_1+a_{11}u)u]+u(1-u-a_{1}v)=0,&~~&x\in\Omega,\\
&\Delta[(d_2+a_{22}v)v]+v(1-a_{2}u-v)=0,&~~&x\in\Omega,\\
&\alpha_1u+\beta_1\frac{\partial u}{\partial \nu}=\alpha_2v+\beta_2\frac{\partial v}{\partial \nu}=0,&~~&x\in\partial\Omega.
\end{aligned}
\right.
\end{align}
We aim to obtain sufficient conditions such that the system \eqref{equ} has no coexisting solutions
under three different boundary conditions,
and establish the parameter ranges for the existence of constant coexisting solutions under Neumann boundary conditions.
Considering that $u$ and $v$ represent species densities,
we focus on the nonnegative classical solution $(u,v)$ of \eqref{equ},
which means that $(u,v)\in(C^{1}(\overline{\Omega})\cap C^{2}(\Omega))^{2}$, $u, v\geq0$ in $\overline{\Omega}$,
and satisfies \eqref{equ} in the pointwise sense.

The remainder of this work is organized as follows.
Section 2 gives some basic preliminaries which implies the strict positivity of
the nontrivial solutions of system \eqref{equ}.
In section 3, based on the boundedness of solutions, we obtain two different parameter ranges for nonexistence of
coexisting solutions under three boundary conditions.
In section 4, we establish the sufficient conditions for the existence of constant coexisting solutions
under Neumann boundary conditions.

\section{Preliminaries}

First of all, we can obtain the positivity of $u$ and $v$ in $\Omega$,
which is crucial in subsequent section.

\begin{prop}\label{prop2.1}
Let $(u,v)$ be a nonnegative classical solution of \eqref{equ}.
Then if $u\not\equiv0$, we have $u>0$ in $\Omega$,
and if $v\not\equiv0$, we have $v>0$ in $\Omega$.
\end{prop}

\begin{proof}
We only prove $u>0$ in $\Omega$ whenever $u\not\equiv0$,
since the positivity of $v$ in $\Omega$ can be proved in a similar way.
Otherwise, there is $x_{0}\in\Omega$ such that $u(x_{0})=\min\limits_{x\in\overline{\Omega}}u(x)=0$.

It follows from the first equation of \eqref{equ} that
$$(d_1+2a_{11}u)\Delta u+2a_{11}|\nabla u|^2+u(1-u-a_1v)=0.$$
Let $$Lu=-(d_1+2a_{11}u)\Delta u-2a_{11}|\nabla u|^2+cu~~~{\rm with}~~~c=u+a_{1}v.$$
Then $$c\geq0~~{\rm and}~~Lu=u\geq0~~{\rm in}~\Omega.$$
So, an application of the strong maximum principle shows that $u$ is constant in $\Omega$,~and thus $u=0$,
a contradiction to $u\not\equiv0$.
This completes the proof.
\end{proof}

\begin{rem}\label{rem2.1}
In the case of Neumann or Robin boundary conditions,
we can get further that $u,v>0$ in $\overline{\Omega}$  by Hopf's boundary lemma.
In fact, for example, considering the case of Robin boundary conditions,
suppose that there is $x_{0}\in\overline{\Omega}$ such that $u(x_{0})=\min\limits_{x\in\overline{\Omega}}u(x)=0$.
If $x_0\in\Omega$, we can directly derive a contradiction by Proposition \ref{prop2.1}.
If $x_0\in\partial\Omega$, then $u(x_0)< u(x)$ for all $x\in\Omega$.
Since $Lu\geq0$ in $\Omega$ and $\Omega$ satisfies the interior ball condition at $x_0\in\partial\Omega$,
it follows from Hopf's boundary lemma that $\frac{\partial u}{\partial\nu}(x_0)<0$.
Hence
$$\alpha_1u(x_0)+\beta_1\frac{\partial u}{\partial\nu}(x_0)<0,$$
a condiction. Therefore, $u>0$ in $\overline{\Omega}$.
\end{rem}

\section{Nonexistence of coexisting steady-state solutions}

In this section, we will discuss the nonexistence of coexisting solutions for system \eqref{equ}
under three different boundary conditions.

\subsection{Neumann boundary condition}

In this subsection, assume that $\alpha_1=\alpha_2=0$, that is,
we consider the following system,
\begin{align}\label{neu}
\left\{
\begin{aligned}
&\Delta[(d_1+a_{11}u)u]+u(1-u-a_{1}v)=0,&~~&x\in\Omega,\\
&\Delta[(d_2+a_{22}v)v]+v(1-a_{2}u-v)=0,&~~&x\in\Omega,\\
&\frac{\partial u}{\partial \nu}=\frac{\partial v}{\partial \nu}=0,&~~&x\in\partial\Omega.
\end{aligned}
\right.
\end{align}
Firstly, we give the following lemma, which indicates that $u$ and $v$ are both bounded.

\begin{lemma}\label{lem3.1}
Suppose that $(u,v)$ is a nonnegative classical solution of \eqref{neu}.
If $u\not\equiv0$, then $u\leq1$ in $\overline{\Omega}$.
Similarly, if $v\not\equiv0$, then $v\leq1$ in $\overline{\Omega}$.
\end{lemma}

\begin{proof}
We only prove the boundedness of $u$,
while the boundedness of $v$ can be similarly obtained, leaving it for interested readers.
Rewrite the first equation of the system \eqref{neu} as follows
\begin{align}\label{3.2}
\left\{
\begin{aligned}
&(d_1+2a_{11}u)\Delta u+2a_{11}|\nabla u|^2+u(1-u-a_{1}v)=0,&~~&x\in\Omega,\\
&\frac{\partial u}{\partial \nu}=0,&~~&x\in\partial\Omega.
\end{aligned}
\right.
\end{align}
Let $f(x)=u(x)(1-u(x)-a_{1}v(x))$, $x\in\Omega$.
Due to Proposition 2.2 in \cite{ln1996}, there exist $x_1,x_2\in\overline{\Omega}$ such that
$$u(x_1)=\max_{x\in\overline{\Omega}}u(x)\triangleq M,~~u(x_2)=\min_{x\in\overline{\Omega}}u(x)\triangleq m,$$
and 
$$f(x_1)\geq0,~~~f(x_2)\leq0,$$ 
that is
$$M(1-M-a_{1}v(x_1))\geq0,~~m(1-m-a_{1}v(x_2))\leq0.$$
Then we have
$$1-a_{1}v(x_2)\leq m\leq u(x)\leq M\leq1-a_{1}v(x_1) \mbox{~~for~all~}x\in\overline{\Omega}.$$
Combining this with the nonnegativity of $v$, we obtain $u\leq1$ in $\overline{\Omega}$.
\end{proof}

\begin{thm}\label{thm3.1}
Suppose that $(u,v)$ is a nonnegative classical solution of \eqref{neu}. If
$$\mbox{(i)}~a_1>1>a_2, d_1>d_2 \mbox{~and~} d_1\geq d_2+2a_{22}$$
or
$$\mbox{(ii)}~a_1<1<a_2, d_1<d_2 \mbox{~and~}d_2\geq d_1+2a_{11},$$
then system \eqref{neu} has no coexisting solutions, that is, at least one species is extinct.
\end{thm}

\begin{proof}
(i) We argue by contradiction. Suppose that $u\not\equiv0$ and $v\not\equiv0$.
It follows from Proposition \ref{prop2.1} that $u,v>0$ in $\Omega$,
which allows us to rewrite system \eqref{neu} as follows,
\begin{align}\label{3.3}
\left\{
\begin{aligned}
&\frac{\Delta[(d_{1}+a_{11}u)u]}{u}=-1+u+a_{1}v,&~~&x\in\Omega,\\
&\frac{\Delta[(d_{2}+a_{22}v)v]}{v}=-1+a_{2}u+v,&~~&x\in\Omega,\\
&\frac{\partial u}{\partial \nu}=\frac{\partial v}{\partial \nu}=0,&~~&x\in\partial\Omega.\\
\end{aligned}
\right.
\end{align}
Let
\begin{align}\label{3.3+1}
w_{1}=d_{1}+a_{11}u\mbox{~~and~~}w_{2}=d_{2}+a_{22}v.
\end{align}
Since $a_1>1>a_2$, we can obtain that $\frac{\Delta (w_{1}u)}{u}>\frac{\Delta (w_{2}v)}{v}$, that is,
\begin{align}\label{3.4}
\frac{u\Delta w_{1}+2\nabla u\cdot\nabla w_{1}+w_{1}\Delta u}{u}>
\frac{v\Delta w_{2}+2\nabla v\cdot\nabla w_{2}+w_{2}\Delta v}{v}~~~\mbox{in}~\Omega.
\end{align}
We procedure the following calculation based on the inequality above,
\begin{align*}
&\mbox{div}\Big[(uv\nabla w_{1}-uv\nabla w_{2}+vw_{1}\nabla u-uw_{2}\nabla v)\frac{u}{v}\Big]\\
=&\big(uv\Delta w_{1}+v\nabla u\cdot\nabla w_{1}+u\nabla v\cdot\nabla w_{1}-uv\Delta w_{2}
-v\nabla u\cdot\nabla w_{2}-u\nabla v\cdot\nabla w_{2}\\
&+vw_{1}\Delta u+v\nabla u\cdot\nabla w_{1}+w_{1}\nabla u\cdot\nabla v-uw_{2}\Delta v
-u\nabla v\cdot\nabla w_{2}-w_{2}\nabla u\cdot\nabla v\big)\frac{u}{v}\\
&+\big(uv\nabla w_{1}-uv\nabla w_{2}+vw_{1}\nabla u-uw_{2}\nabla v\big)
\cdot\Big(\frac{1}{v}\nabla u-\frac{u}{v^2}\nabla v\Big)\\
>&\big(u\nabla v\cdot\nabla w_{1}-v\nabla u\cdot\nabla w_{2}
+w_{1}\nabla u\cdot\nabla v-w_{2}\nabla u\cdot\nabla v\big)\frac{u}{v}\\
&+\big(uv\nabla w_{1}-uv\nabla w_{2}+vw_{1}\nabla u-uw_{2}\nabla v\big)
\cdot\Big(\frac{1}{v}\nabla u-\frac{u}{v^2}\nabla v\Big)\\
=&|\nabla u|^{2}\big(a_{11}u+w_{1}\big)+|\nabla v|^{2}\big(a_{22}v+w_{2}\big)\frac{u^2}{v^2}\\
&+\nabla u\cdot\nabla v
\Big[(a_{11}u+w_1)\frac{u}{v}+(-a_{22}v-w_2)\frac{u}{v}+(-a_{11}u-w_1-a_{22}v-w_2)\frac{u}{v}\Big]\\
=&|\nabla u|^{2}\big(d_{1}+2a_{11}u\big)+|\nabla v|^{2}\big(d_{2}+2a_{22}v\big)\frac{u^2}{v^2}
-2(d_{2}+2a_{22}v)\frac{u}{v}\nabla u\cdot\nabla v\\
=&|\nabla u|^{2}\big(d_{1}+2a_{11}u-d_{2}-2a_{22}v\big)
+\left(\sqrt{d_{2}+2a_{22}v}\nabla u-\frac{u}{v}\sqrt{d_{2}+2a_{22}v}\nabla v\right)^2
\end{align*}
Due to $d_1>d_2$, $d_1\geq d_2+2a_{22}$ and Lemma \ref{lem3.1},
we have 
$$d_{1}+2a_{11}u-d_{2}-2a_{22}v\geq d_{1}+2a_{11}u-d_{2}-2a_{22}\geq0 \mbox{~~~in~} \Omega.$$ 
Thus,
\begin{align}\label{3.5}
\mbox{div}\Big[(uv\nabla w_{1}-uv\nabla w_{2}+vw_{1}\nabla u-uw_{2}\nabla v)\frac{u}{v}\Big]>0.
\end{align}

On the other hand, we can see from Neumann boundary conditions that
\begin{align*}
&\int_\Omega\mbox{div}\left[\big(uv\nabla w_{1}-uv\nabla w_{2}+v w_{1}\nabla u-u w_{2}\nabla v\big)
\frac{u}{v}\right]\mathrm{d}x\\
=&\int_{\partial\Omega}\left(uv\frac{\partial w_{1}}{\partial\nu}-uv\frac{\partial w_{2}}{\partial\nu}
+v w_{1}\frac{\partial u}{\partial\nu}-u w_{2}\frac{\partial v}{\partial\nu}\right)\frac{u}{v}\mathrm{d}S\\
=&\int_{\partial\Omega}\left(a_{11}uv\frac{\partial u}{\partial\nu}-a_{22}uv\frac{\partial v}{\partial\nu}
+v w_{1}\frac{\partial u}{\partial\nu}-u w_{2}\frac{\partial v}{\partial\nu}\right)\frac{u}{v}\mathrm{d}S\\
=&0,
\end{align*}
which is contradict to \eqref{3.5}. Therefore, $u\equiv0$ or $v\equiv0$.

(ii) In this case, we can see that $\frac{\Delta (w_{1}u)}{u}<\frac{\Delta (w_{2}v)}{v}$ due to $a_1<1<a_2$,
which can be expanded into the following form
\begin{align}\label{3.6}
\frac{u\Delta w_{1}+2\nabla u\cdot\nabla w_{1}+w_{1}\Delta u}{u}<
\frac{v\Delta w_{2}+2\nabla v\cdot\nabla w_{2}+w_{2}\Delta v}{v}~~~\mbox{in}~\Omega,
\end{align}
where $w_1,w_2$ are defined in \eqref{3.3+1}.
Similar to the calculations of (i), we consider the process as follows
\begin{align*}
&\mbox{div}\Big[(uv\nabla w_{1}-uv\nabla w_{2}+vw_{1}\nabla u-uw_{2}\nabla v)\frac{v}{u}\Big]\\
<&\big(u\nabla v\cdot\nabla w_{1}-v\nabla u\cdot\nabla w_{2}
+w_{1}\nabla u\cdot\nabla v-w_{2}\nabla u\cdot\nabla v\big)\frac{v}{u}\\
&+\big(uv\nabla w_{1}-uv\nabla w_{2}+vw_{1}\nabla u-uw_{2}\nabla v\big)
\cdot\Big(-\frac{v}{u^2}\nabla u+\frac{1}{u}\nabla v\Big)\\
=&|\nabla u|^{2}\big(-a_{11}u-w_{1}\big)\frac{v^2}{u^2}+|\nabla v|^{2}\big(-a_{22}v-w_{2}\big)\\
&+\nabla u\cdot\nabla v
\Big[(a_{11}u+w_1)\frac{v}{u}+(-a_{22}v-w_2)\frac{v}{u}+(a_{11}u+w_1+a_{22}v+w_2)\frac{v}{u}\Big]\\
=&-|\nabla u|^{2}\big(d_{1}+2a_{11}u\big)\frac{v^2}{u^2}-|\nabla v|^{2}\big(d_{2}+2a_{22}v\big)
+2(d_{1}+2a_{11}u)\frac{v}{u}\nabla u\cdot\nabla v\\
=&|\nabla v|^{2}\big(-d_{2}-2a_{22}v+d_{1}+2a_{11}u\big)
-\left(\frac{v}{u}\sqrt{d_{1}+2a_{11}u}\nabla u-\sqrt{d_{1}+2a_{11}u}\nabla v\right)^2
\end{align*}
Due to $d_1<d_2$, $d_2\geq d_1+2a_{11}$ and Lemma \ref{lem3.1},
we have 
$$-d_{2}-2a_{22}v+d_{1}+2a_{11}u\leq-d_{2}-2a_{22}v+d_{1}+2a_{11}\leq0 \mbox{~~~in~} \Omega.$$ 
Thus
\begin{align}\label{3.7}
\mbox{div}\Big[(uv\nabla w_{1}-uv\nabla w_{2}+vw_{1}\nabla u-uw_{2}\nabla v)\frac{v}{u}\Big]<0.
\end{align}

Now, based on the boundary conditions, we can still obtain the following result
\begin{align*}
&\int_\Omega
\mbox{div}\left[\big(uv\nabla w_{1}-uv\nabla w_{2}+v w_{1}\nabla u-u w_{2}\nabla v\big)\frac{v}{u}\right]\mathrm{d}x=0,
\end{align*}
which leads to a contradiction. This completes the proof.
\end{proof}

\subsection{Dirichlet boundary condition}

In this subsection, we consider the case of $\beta_1=\beta_2=0$, which leads to the following system,
\begin{align}\label{dir}
\left\{
\begin{aligned}
&\Delta[(d_1+a_{11}u)u]+u(1-u-a_{1}v)=0,&~~&x\in\Omega,\\
&\Delta[(d_2+a_{22}v)v]+v(1-a_{2}u-v)=0,&~~&x\in\Omega,\\
&u=v=0,&~~&x\in\partial\Omega.
\end{aligned}
\right.
\end{align}
Now, we can also deduce that $u$ and $v$ are both bounded.

\begin{lemma}\label{lem3.2}
Suppose that $(u,v)$ is a nonnegative classical solution of \eqref{dir}.
If $u\not\equiv0$, then $u\leq1$ in $\overline{\Omega}$.
Similar, if $v\not\equiv0$, then $v\leq1$ in $\overline{\Omega}$.
\end{lemma}

\begin{proof}
We only need to prove the boundedness of $v$.
Proposition \ref{prop2.1} implies that $v>0$ in $\Omega$ if $v\not\equiv0$.
The second equation of system \eqref{dir} can be transformed into the following form,
\begin{align}\label{}
\left\{
\begin{aligned}
&(d_2+2a_{22}v)\Delta v+2a_{22}|\nabla v|^2+v(1-a_2u-v)=0,&~~&x\in\Omega,\\
&v=0,&~~&x\in\partial\Omega.
\end{aligned}
\right.
\end{align}
Suppose on the contrary that there exists a point $x_{0}$ such that $v(x_{0})=\max_{\overline{\Omega}}v(x)>1$.
Obviously, $x_{0}\in \Omega$.
So we have 
$$\Delta v(x_{0})\leq0,~ \nabla v(x_0)=0 \mbox{~and~} v(x_{0})(1-a_{2}u(x_{0})-v(x_{0}))\geq0.$$
Since $v>0$ in $\Omega$, we obtain that $1-a_{2}u(x_{0})-v(x_{0})\geq0$,
that is, $1\geq a_{2}u(x_{0})+v(x_{0})$, a contradiction.
\end{proof}

\begin{thm}\label{thm3.2}
Suppose that $(u,v)$ is a nonnegative classical solution of \eqref{dir}. If
$$\mbox{(i)}~a_1>1>a_2, d_1>d_2 \mbox{~and~} d_1\geq d_2+2a_{22}$$
or
$$\mbox{(ii)}~a_1<1<a_2, d_1<d_2 \mbox{~and~}d_2\geq d_1+2a_{11},$$
then system \eqref{dir} has no coexisting solutions.
\end{thm}

\begin{proof}
(i) According to the proof of Theorem \ref{thm3.1}-(i), the following inequality still holds 
$$\mbox{div}\Big[(uv\nabla w_{1}-uv\nabla w_{2}+vw_{1}\nabla u-uw_{2}\nabla v)\frac{u}{v}\Big]>0,$$
where $w_1=d_1+a_{11}u$ and $w_2=d_2+a_{22}v$.

Now, we consider the following integral
\begin{align*}
&\int_\Omega\mbox{div}\left[\Big(uv\nabla w_{1}-uv\nabla w_{2}+v w_{1}\nabla u-u w_{2}\nabla v\Big)
\frac{u}{v}\right]\mathrm{d}x\\
=&\int_{\partial\Omega}\left(uv\frac{\partial w_{1}}{\partial\nu}-uv\frac{\partial w_{2}}{\partial\nu}
+v w_{1}\frac{\partial u}{\partial\nu}-u w_{2}\frac{\partial v}{\partial\nu}\right)\frac{u}{v}\mathrm{d}S\\
=&\int_{\partial\Omega}\left(a_{11}uv\frac{\partial u}{\partial\nu}-a_{22}uv\frac{\partial v}{\partial\nu}
+v w_{1}\frac{\partial u}{\partial\nu}-u w_{2}\frac{\partial v}{\partial\nu}\right)\frac{u}{v}\mathrm{d}S\\
=&\int_{\partial\Omega}\left(a_{11}u^2\frac{\partial u}{\partial\nu}-a_{22}u^2\frac{\partial v}{\partial\nu}
+uw_{1}\frac{\partial u}{\partial\nu}- \left(d_2\frac{u^2}{v}+a_{22}u^2\right)\frac{\partial v}{\partial\nu}\right)\mathrm{d}S.
\end{align*}
It is easy to see that the function $\frac{u^2}{v}$ in the last term of the integrand doesn't make sense on $\partial\Omega$.
So in such a case we cannot make calculations directly. Let
$$\Omega_{\varepsilon}=\{x\in\Omega\big|{\rm dist}(x,\partial\Omega)>\varepsilon\}~~{\rm for~any~small~}\varepsilon>0.$$
Since $(u,v)\in(C^{1}(\overline{\Omega})\cap C^{2}(\Omega))^{2}$,
we observe that
$$\big(uv\nabla w_{1}-uv\nabla w_{2}+v w_{1}\nabla u-u w_{2}\nabla v\big)\frac{u}{v}
\in C^{1}(\overline{\Omega_{\varepsilon}}).$$
Then divergence theorem implies that
\begin{align*}
&\int_\Omega\mbox{div}\left[\big(uv\nabla w_{1}-uv\nabla w_{2}+v w_{1}\nabla u-u w_{2}\nabla v\big)
\frac{u}{v}\right]\mathrm{d}x\\
=&\lim _{\varepsilon\to0}\int_{\Omega_\varepsilon}
\mbox{div}\left[\big(uv\nabla w_{1}-uv\nabla w_{2}+v w_{1}\nabla u-u w_{2}\nabla v\big)\frac{u}{v}\right]\mathrm{d}x\\
=&\lim _{\varepsilon\to0}\int_{\partial\Omega_\varepsilon}
\left(a_{11}u^2\frac{\partial u}{\partial\nu}-a_{22}u^2\frac{\partial v}{\partial\nu}+uw_{1}\frac{\partial u}{\partial\nu}
-\left(d_2\frac{u^2}{v}+a_{22}u^2\right)\frac{\partial v}{\partial\nu}\right)\mathrm{d}S.
\end{align*}
Let
\begin{align*}
&I_{1}(\varepsilon)=\int_{\partial\Omega_\varepsilon}
\left(a_{11}u^2\frac{\partial u}{\partial\nu}-a_{22}u^2\frac{\partial v}{\partial\nu}+uw_{1}\frac{\partial u}{\partial\nu}
-a_{22}u^2\frac{\partial v}{\partial\nu}\right)\mathrm{d}S,\\
&I_{2}(\varepsilon)=\int_{\partial\Omega_\varepsilon}
d_2\frac{u^2}{v}\frac{\partial v}{\partial\nu}\mathrm{d}S.
\end{align*}
Obviously, $I_{1}(\varepsilon)$ approaches zero as $\varepsilon\rightarrow 0$ in terms of Dirichlet boundary conditions.
In order to deal with the term $I_{2}(\varepsilon)$, we write
$$V=\left\{\varphi(x)\in C^{1}(\overline{\Omega})\big|\varphi(x)>0,x\in\Omega;~\varphi|_{\partial\Omega}=0;
~\frac{\partial \varphi}{\partial\nu}\Big|_{\partial\Omega}<0\right\}.$$
By Hopf's Lemma, we have $\frac{\partial u(x_{0})}{\partial \nu}<0$ and $\frac{\partial v(x_{0})}{\partial \nu}<0$
for any $x_{0}\in\partial\Omega$, and thus $u\in V$ and $v\in V$.
We now define
\begin{align*}
g(x):=\left\{
\begin{aligned}
&\frac{u(x)}{v(x)},~~~~~~~~~~~~~~~&~~x\in \Omega, \\
&\frac{\partial u(x)}{\partial\nu}\big/\frac{\partial v(x)}{\partial\nu},~~~~&~~ x\in \partial\Omega.\\
\end{aligned}
\right.
\end{align*}
Then Lemma 2.4 in \cite{yz2013} shows that $g(x)\in C\big(\overline{\Omega},(0,+\infty)\big)$.
Therefore, we conclude that $I_{2}(\varepsilon)$ also approaches zero as $\varepsilon\rightarrow 0$.
Consequently,
\begin{align}\label{3.8}
\int_\Omega\mbox{div}
\left[\big(uv\nabla w_{1}-uv\nabla w_{2}+v w_{1}\nabla u-u w_{2}\nabla v\big)\frac{u}{v}\right]\mathrm{d}x=0
\end{align}
thanks to Lebesgue dominated convergence theorem and Dirichlet boundary conditions,
a contradiction.

(ii) On one hand, it is obvious to see from Theorem \ref{thm3.1}-(ii) that
\begin{align}\label{3.9}
\mbox{div}\Big[(uv\nabla w_{1}-uv\nabla w_{2}+vw_{1}\nabla u-uw_{2}\nabla v)\frac{v}{u}\Big]<0,
\end{align}
due to $d_1<d_2$, $d_2\geq d_1+2a_{11}$ and Lemma \ref{lem3.2}.

On the other hand, it follows from the analysis above that the following equality still holds
\begin{align}\label{3.10}
\int_\Omega\mbox{div}
\left[\big(uv\nabla w_{1}-uv\nabla w_{2}+v w_{1}\nabla u-u w_{2}\nabla v\big)\frac{v}{u}\right]\mathrm{d}x=0.
\end{align}
Therefore, a contradiction can also be obtained, and the conclusion is valid.
\end{proof}

\subsection{Robin boundary condition}

In this subsection, we consider the following system
\begin{align}\label{rob}
\left\{
\begin{aligned}
&\Delta[(d_1+a_{11}u)u]+u(1-u-a_{1}v)=0,&~~&x\in\Omega,\\
&\Delta[(d_2+a_{22}v)v]+v(1-a_{2}u-v)=0,&~~&x\in\Omega,\\
&\alpha_1u+\beta_1\frac{\partial u}{\partial \nu}=\alpha_2v+\beta_2\frac{\partial v}{\partial \nu}=0,&~~&x\in\partial\Omega,
\end{aligned}
\right.
\end{align}
where $\alpha_i>0$, $\beta_i>0$, $i=1,2$.

\begin{lemma}
Suppose that $(u,v)$ is a nonnegative classical solution of \eqref{rob}.
If $u\not\equiv0$, then $u\leq1$ in $\overline{\Omega}$.
Similar, if $v\not\equiv0$, then $v\leq1$ in $\overline{\Omega}$.
\end{lemma}

\begin{proof}
We only prove the previous statement.
First, we can obtain that
\begin{align}\label{}
\left\{
\begin{aligned}
&(d_1+2a_{11}u)\Delta u+2a_{11}|\nabla u|^2+u(1-u-a_{1}v)=0,&~~&x\in\Omega,\\
&\alpha_1u+\beta_1\frac{\partial u}{\partial \nu}=0,&~~&x\in\partial\Omega.
\end{aligned}
\right.
\end{align}
Suppose that there exists a point $x_0\in\overline{\Omega}$,
such that $u(x_0)=\max_{x\in\overline{\Omega}}u(x)>1$.

(i) If $x_0\in \Omega$, we have $\Delta u(x_{0})\leq0$, $\nabla v(x_0)=0$ and $u(x_{0})(1-u(x_{0})-a_1v(x_{0}))\geq0$.
Since $u>0$ in $\Omega$, we obtain that $1-u(x_{0})-a_1v(x_{0})\geq0$,
that is, $1\geq u(x_{0})+a_1v(x_{0})$, a contradiction.

(ii) If $x_0\in \partial\Omega$, then $u(x_0)>u(x)$ for all $x\in\Omega$.
Thus we can see that $\frac{\partial u}{\partial \nu}(x_0)\geq0$.
Since $u>0$ in $\overline{\Omega}$ by Remark \ref{rem2.1}, we have 
$$\alpha_1u(x_0)+\beta_1\frac{\partial u}{\partial \nu}(x_0)>0,$$
a contradiction.
\end{proof}

\begin{thm}\label{thm3.3}
Suppose that $(u,v)$ is a nonnegative classical solution of \eqref{rob}. If
$$\mbox{(i)}~a_1>1>a_2, d_1>d_2, d_1\geq d_2+2a_{22}, \alpha_1\geq\beta_1 \mbox{~and~} \alpha_2\leq\beta_2$$
or
$$\mbox{(ii)}~a_1<1<a_2, d_1<d_2, d_2\geq d_1+2a_{11} , \alpha_1\leq\beta_1 \mbox{~and~} \alpha_2\geq\beta_2,$$
then system \eqref{rob} has no coexisting solutions.
\end{thm}

\begin{proof}
We only prove the first case. Now, we can still obtain
$$\mbox{div}\Big[(uv\nabla w_{1}-uv\nabla w_{2}+vw_{1}\nabla u-uw_{2}\nabla v)\frac{u}{v}\Big]>0,$$
where $w_1=d_1+a_{11}u$ and $w_2=d_2+a_{22}v$.

Notice that $\frac{\partial u}{\partial \nu}=-\frac{\alpha_1}{\beta_1}u$,
$\frac{\partial v}{\partial \nu}=-\frac{\alpha_2}{\beta_2}v$ in $\partial\Omega$.
We consider the following integral
\begin{align*}
&\int_\Omega\mbox{div}\left[\big(uv\nabla w_{1}-uv\nabla w_{2}+v w_{1}\nabla u-u w_{2}\nabla v\big)
\frac{u}{v}\right]\mathrm{d}x\\
=&\int_{\partial\Omega}\left(a_{11}uv\frac{\partial u}{\partial\nu}-a_{22}uv\frac{\partial v}{\partial\nu}
+v w_{1}\frac{\partial u}{\partial\nu}-u w_{2}\frac{\partial v}{\partial\nu}\right)\frac{u}{v}\mathrm{d}S\\
=&\int_{\partial\Omega}\left(-a_{11}\frac{\alpha_1}{\beta_1}u^2v+a_{22}\frac{\alpha_2}{\beta_2}uv^2
-\frac{\alpha_1}{\beta_1}uv w_{1}+\frac{\alpha_2}{\beta_2}uv w_{2}\right)\frac{u}{v}\mathrm{d}S\\
=&\int_{\partial\Omega}\left(-a_{11}\frac{\alpha_1}{\beta_1}u^3+a_{22}\frac{\alpha_2}{\beta_2}u^2 v
-\frac{\alpha_1}{\beta_1}u^2 w_{1}+\frac{\alpha_2}{\beta_2}u^2 w_{2}\right)\mathrm{d}S\\
=&\int_{\partial\Omega}u^2\left(-a_{11}\frac{\alpha_1}{\beta_1}u+a_{22}\frac{\alpha_2}{\beta_2} v
-\frac{\alpha_1}{\beta_1} (d_1+a_{11}u)+\frac{\alpha_2}{\beta_2} (d_2+a_{22}v)\right)\mathrm{d}S\\
=&\int_{\partial\Omega}u^2\left(-2a_{11}\frac{\alpha_1}{\beta_1}u+2a_{22}\frac{\alpha_2}{\beta_2} v
-d_1\frac{\alpha_1}{\beta_1}+d_2\frac{\alpha_2}{\beta_2} \right)\mathrm{d}S\\
\leq&\int_{\partial\Omega}u^2\left(-2a_{11}\frac{\alpha_1}{\beta_1}u+2a_{22}\frac{\alpha_2}{\beta_2}
-d_1\frac{\alpha_1}{\beta_1}+d_2\frac{\alpha_2}{\beta_2} \right)\mathrm{d}S\\
\leq&\int_{\partial\Omega}u^2\left(-2a_{11}\frac{\alpha_1}{\beta_1}u+2a_{22}-d_1+d_2\right)\mathrm{d}S\\
\leq&0,
\end{align*}
due to $\alpha_1\geq\beta_1$, $\alpha_2\leq\beta_2$ and $d_1\geq d_2+2a_{22}$, a contradiction.

%
\end{proof}

\section{Existence of coexisting solutions}

In this section, we will investigate the existence of coexisting solutions for system \eqref{neu}.

\begin{thm}\label{thm4.1}
Let $a_1<1$ and $a_2<1$. Suppose that $(u,v)$ is a nonnegative classical solution of \eqref{neu}.
If $u\not\equiv0$ and $v\not\equiv0$, then
\begin{align}\label{4.1}
u\equiv\frac{1-a_1}{1-a_1a_2}~~~\mbox{and}~~~v\equiv\frac{1-a_2}{1-a_1a_2}.
\end{align}
\end{thm}

\begin{proof}
By Lemma \ref{lem3.1}, there exist $x_1,x_2\in\overline{\Omega}$ such that
$$u(x_1)=\max_{x\in\overline{\Omega}}u(x)\triangleq M,~~u(x_2)=\min_{x\in\overline{\Omega}}u(x)\triangleq m,$$
and
\begin{align}\label{4.2}
1-a_{1}v(x_2)\leq m\leq u(x)\leq M\leq1-a_{1}v(x_1) \mbox{~~for~all~}x\in\overline{\Omega}.
\end{align}
Similarly, there exist $x_3,x_4\in\overline{\Omega}$ such that
$$v(x_3)=\max_{x\in\overline{\Omega}}v(x)\triangleq M',~~v(x_4)=\min_{x\in\overline{\Omega}}v(x)\triangleq m',$$
and
\begin{align}\label{4.3}
1-a_{2}u(x_4)\leq m'\leq v(x)\leq M'\leq1-a_{2}u(x_3) \mbox{~~for~all~}x\in\overline{\Omega}.
\end{align}
We first prove $u\equiv\frac{1-a_1}{1-a_1a_2}$.
Combining \eqref{4.2} with \eqref{4.3}, we obtain
\begin{align}
m\geq1-a_{1}v(x_2)\geq1-a_{1}(1-a_{2}u(x_3))\geq1-a_1+a_1a_2m,
\end{align}
and
\begin{align}
M\leq1-a_{1}v(x_1)\leq1-a_{1}(1-a_{2}u(x_4))\leq1-a_1+a_1a_2M.
\end{align}
Notice that since $a_1<1$ and $a_2<1$, we have
\begin{align}\label{4.4}
\frac{1-a_1}{1-a_1a_2}\leq m\leq u(x)\leq M\leq\frac{1-a_1}{1-a_1a_2}.
\end{align}
Hence
$$u\equiv\frac{1-a_1}{1-a_1a_2}.$$
Similarly, it can be proved that
\begin{align}\label{4.5}
\frac{1-a_2}{1-a_1a_2}\leq m'\leq v(x)\leq M'\leq\frac{1-a_2}{1-a_1a_2}.
\end{align}
Consequently,
$$v\equiv\frac{1-a_2}{1-a_1a_2}.$$
This completes the proof.
\end{proof}

\section*{Acknowledgments}
This work was funded by the Natural Science Foundation of Shandong Province (ZR2021QA103, ZR2021MA016).


\end{document}